\documentclass{daj}

\usepackage{amssymb,amsmath,amscd,amsfonts,amsthm,mathrsfs,color,booktabs}
\usepackage{graphicx}
\usepackage{tikz}

\newcommand\numberthis{\addtocounter{equation}{1}\tag{\theequation}} 

\newcommand{\defn}{\emph}
\newcommand{\eb}{\partial}
\newcommand{\ebopt}{\partial^*}
\newcommand{\vb}{\partial_v}
\newcommand{\vbopt}{\partial_v^*}
\newcommand{\hist}{\mathbf}
\newcommand{\push}[2]{{#1}_\hist{#2}}
\newcommand{\cbound}[1]{C_{|\hist x|}}
\newcommand{\cube}[1]{\mathcal Q_{#1}}
\newcommand{\lineseg}[2]{[0,{#1}_{#2}]}

\DeclareMathOperator*{\E}{\mathbb E}

\DeclareMathOperator{\vol}{vol}

\newcommand{\bb}[1]{\mathbb{#1}}

\newtheorem{theorem}{Theorem}
\newtheorem{proposition}[theorem]{Proposition}
\newtheorem{lemma}[theorem]{Lemma}

\newtheorem{question}{Question}
\newtheorem*{question*}{Question}

\theoremstyle{definition}

\theoremstyle{remark}
\newtheorem{remark}[theorem]{Remark}

\dajAUTHORdetails{%
  title = {Isoperimetry in Integer Lattices}, 
  author = {Ben Barber and Joshua Erde},
  plaintextauthor = {Ben Barber and Joshua Erde},
}   

\dajEDITORdetails{%
   year={2018},
   volume={XX},
   number={7},
   received={24 July 2017},   
   revised={9 March 2018},    
   published={20 April 2018},  
   doi={10.19086/da.3555},       
}   

\begin{document}

\begin{frontmatter}[classification=text]


\author[benbarber]{Ben Barber}
\author[joshuaerde]{Joshua Erde\thanks{Supported by the Alexander von Humboldt Foundation}}

\begin{abstract}
The edge isoperimetric problem for a graph $G$ is to determine, for each $n$, the minimum number of edges leaving any set of $n$ vertices.
In general this problem is NP-hard, but exact solutions are known in some special cases, for example when $G$ is the usual integer lattice.
We solve the edge isoperimetric problem asymptotically for every Cayley graph on $\mathbb Z^d$.
The near-optimal shapes that we exhibit are zonotopes generated by line segments corresponding to the generators of the Cayley graph.
\end{abstract}
\end{frontmatter}

\section{Introduction}

For every space equipped with notions of size and boundary of subsets there is a corresponding isoperimetric problem: how small can the boundary be over all subsets of a given size?
For example, the classical isoperimetric theorem states that the measurable subset of $\mathbb R^d$ with  minimum boundary for a given volume is an appropriate scaling of the unit ball.
Isoperimetric problems can also be posed for graphs, where they are closely related to the phenomenon of expansion.
Isoperimetric inequalities measure how easy it is to separate a set of vertices from the rest of the graph, which in turn can be related to the mixing time of Markov chains \cite{JS89}, or performance of error correcting codes \cite{SS96}.
There are two commonly studied isoperimetric problems on graphs, corresponding to two natural definitions of the boundary of a set of vertices.
It will be convenient to state the definitions for directed graphs; for undirected graphs consider the directed graph obtained by replacing each edge by a pair of edges oriented in opposite directions.
Given a directed graph $G$, the \defn{edge boundary} of a set $S \subseteq V(G)$ is
\[
\eb(S) = \eb_G(S)= |\{(u,v) \in E(G) \, : \, u \in S,  v \in V(G) \setminus S\}|,
\]
and the \defn{vertex boundary} is
\[
\vb(S) = \partial_{v,G}(S) = | \{ v \in V(G) \setminus S \, : \, \text{there exists } u \in S \text{ such that } (u,v) \in E(G)\}|.
\]
That is, the edge boundary is the number of edges leaving $S$, and the vertex boundary is the number of vertices we can reach by following these edges.
Thus we always have the inequalities
\begin{equation} \label{differ-by-a-constant}
\vb(S) \leq \eb(S) \leq \Delta_{\text{in}}(G)\vb(S),
\end{equation}
where $\Delta_{\text{in}}(G)$ is the maximum in-degree of any vertex in $G$.
We write
\[
\ebopt(n) = \min_{S \subseteq V(G) \, : \, |S| = n} \eb(S), \qquad\qquad 
\vbopt(n) = \min_{S \subseteq V(G) \, : \, |S| = n} \vb(S)
\]
for the minimum size of the edge or vertex boundary over all subsets of size $n$.
The edge (respectively, vertex) isoperimetric problem on $G$ is to determine the function $\ebopt$ (respectively, $\vbopt$).
Solving either of these problems for a general graph $G$ is NP-hard \cite{GJS1976,BV2009}, but results are known in several special cases where $G$ has a lot of structure.
The approximate shape of the optimal sets for some families of graphs are listed in Table~\ref{table}.

\begin{table}[t]
\begin{center}
\begin{tabular}{ccc}
\toprule
Graph & Edge-optimal shapes & Vertex-optimal shapes \\
\midrule
$\mathcal Q_d$ & subcubes \cite{H1964,L1964,B1967,H1976} & Hamming balls \cite{H1966} \\
$(\mathbb Z^d, l_1)$ & cubes ($l_\infty$-balls) \cite{BL19912} & cross-polytopes ($l_1$-balls) \cite{WW1977} \\
$(\mathbb Z^d, l_\infty)$ &  & cubes \cite{RV2012} \\
\bottomrule
\end{tabular}
\caption{Approximate shapes of optimal sets for isoperimetric problems on certain families of graphs.} \label{table}
\end{center}
\end{table}

The $d$-dimensional hypercube $\cube d$ is the graph on vertex set $\{0, 1\}^d$ with edges between those pairs of binary strings that differ in a single coordinate. With a coding theory application in mind, the edge isoperimetric problem was solved by Harper \cite{H1964}, Lindsey \cite{L1964}, Bernstein \cite{B1967} and Hart \cite{H1976}.
The optimal sets include $k$-dimensional subcubes obtained by fixing $d-k$ coordinates and allowing the rest to take all $2^k$ possible values.
The corresponding vertex isoperimetric problem was solved by Harper \cite{H1966}.
The optimal sets include Hamming balls: for each $w$, the sets of strings with at most $w$ coordinates equal to $1$. This illustrates a typical feature of isoperimetric problems on graphs: the optimal shapes for each type of boundary are very different.

Isoperimetric problems have also been studied for many grid-like graphs. 
Let $(\mathbb Z^d, l_1)$ be the graph on vertex set $\mathbb Z^d$ with edges between pairs of vertices at $l_1$-distance~$1$.
Wang and Wang \cite{WW1977} showed that the optimal sets for the vertex isoperimetric problem on this graph include $l_1$-balls consisting of all vertices with $l_1$-norm at most $w$.
They also proved the same result for the restriction of this graph to $(\bb{N}^d,l_1)$.
Bollob\'as and Leader \cite{BL1991} showed that these sets also remain optimal when restricted to the finite grids $([k]^d,l_1)$ (where we write $[k] = \{1, \ldots, k\}$) generalising Harper's result for $\mathcal Q_d$.
The edge isoperimetric problem on $(\bb{Z}^d,l_1)$ was solved by Bollob\'as and Leader \cite{BL19912};
the optimal shapes include cubes.
(An asymptotic solution follows from the Loomis--Whitney inequality \cite{LW49}.)
More recently, Radcliffe and Veomett \cite{RV2012} solved the vertex isoperimetric problem in the $l_{\infty}$-grid $(\bb{Z}^d,l_{\infty})$, where two points are adjacent if they are at $l_{\infty}$-distance~$1$. 
Again, the optimal shapes include cubes.

For each of the preceding results the authors solved the isoperimetric problem exactly.
In fact, they found an ordering $v_1, v_2, \ldots$ of the vertex set such that, for each $n$, the set $\{v_1, \ldots, v_n\}$ has boundary of minimum size.
These orderings remain consistent as the dimension $d$ varies, in the following sense. Write $G_d$ for either $\cube d$, $(\mathbb Z^d, l_1)$ or $(\bb{Z}^d,l_{\infty})$.
Then viewing $G_d$ as a subgraph of $G_{d+1}$ in the natural way, the optimal order for $G_{d+1}$ restricts to an optimal ordering for $G_d$.
This allows each of the preceding results to be proved using `compression' techniques: for more on compressions see \cite{BL1991} or \cite{F1987}.

 Bollob\'as and Leader \cite{BL19912} also considered the edge isoperimetric problem on $([k]^d,l_1)$. Here, when the number of vertices being considered is a large fraction of the size of the grid, the character of the optimal sets changes as edge effects come into play: for example, the half-grid $[k]^{d-1} \times [k/2]$ has a smaller edge boundary in $([k]^d,l_1)$ than a cube containing the same number of points. More seriously, the transition between such qualitatively different optimal sets is not smooth. In particular, there is no ordering of the vertex set for which every initial segment is optimal; this seems to rule out the use of compression techniques.

Instead, Bollob\'as and Leader used the following strategy.
They first approximated the edge isoperimetric problem in $([k]^d,l_1)$ by a continuous problem concerning projections of subsets of $[0,1]^d$. They then solved this problem exactly, obtaining an approximate solution to the original edge isoperimetric theorem. The same strategy has been applied by Harper \cite{H1999} to the vertex isoperimetry problem in the graph on $[k]^d$ with edges between points that differ in exactly one coordinate (those at distance $1$ in the \defn{Hamming metric}).
Again, the optimal solutions are not nested.

Our starting point is the following question:

\begin{question*}
What is the solution to the edge isoperimetric problem on $(\bb{Z}^d,l_{\infty})$?
\end{question*}

An answer to this question would fill the gap in Table~\ref{table}.
In fact we will prove a more general result: we will solve the edge isoperimetric problem asymptotically for every Cayley graph on $\mathbb Z^d$.

\subsection{Cayley graphs}

Let $G$ be a group and let $\mathcal U$ be a generating set for $G$ that does not contain the identity.
The (\defn{directed}) \defn{Cayley graph} $G_{\mathcal U}$ has vertex set $G$ and edge set $\{(g, ug) : g \in G, u \in \mathcal U\}$.
We shall always take $G$ to be $(\mathbb Z^d, +)$ for some $d$ and $\mathcal U = \{u_1, \ldots, u_k\}$ to be a finite set of non-zero vectors.
Then $G_{\mathcal U}$ is a lattice-like graph on $\mathbb Z^d$ in which the neighbourhood of the origin is $\mathcal U$ and the neighbourhood of each other vertex is obtained by translation.
This construction includes both families of lattice graphs considered earlier:
taking $\mathcal U = \{(\pm1, 0, \ldots, 0), \ldots, (0, \ldots, 0, \pm 1)\}$ produces $(\mathbb Z^d, l_1)$; 
taking $\mathcal U = \{-1, 0, 1\}^d \setminus \{(0,\ldots,0)\}$ produces $(\bb{Z}^d,l_{\infty})$.

For subsets $A, B$ of any abelian group, the \defn{sumset} or \defn{Minkowski sum} of $A$ and $B$ is
\[
 A + B = \{a + b \,:\, a \in A, b \in B\}.
\]
We also write
\[
nA = \underbrace{A + \cdots + A}_{n \text{ times}}
\] 
for an iterated sumset.
This is in general distinct from the dilation
\[
n \cdot A = \{na : a \in A\},
\]
but if $n$ is a positive integer and $A$ is a convex subset of $\mathbb R^d$ then the two notions coincide.

\begin{theorem} \label{t:main}
Let $\mathcal U = \{u_1, \ldots, u_k\}$ be a finite set of non-zero vectors that generate $\mathbb Z^d$ as a group.
Let $Z_0$ be the sumset $\{0, u_1\} + \{0, u_2\} +  \cdots + \{0, u_k\}$ and let $Z$ be the convex hull of $Z_0$ in $\mathbb R^d$.
For every $\delta > 0$, there is an $n_0 = n_0(\delta, \mathcal U)$ such that, for every $n \geq n_0$, 
\[
(1-\delta) d \vol(Z)^{1/d} n^{1-1/d} \leq \ebopt_{G_{\mathcal U}}(n) \leq (1+\delta) d \vol(Z)^{1/d} n^{1-1/d}.
\]
The upper bound is witnessed by intersections of scaled copies of $Z$ with $\mathbb Z^d$.
\end{theorem}

The set $Z$ is a `zonotope' in $\mathbb R^d$; see Remark~\ref{zonotopes} in Section~\ref{s:continuous-edge}.
If $\mathcal U$ fails to generate $\mathbb Z^d$ as a group, then $G_{\mathcal U}$ breaks into several connected components, each isomorphic to $G_{\mathcal U'}$ for some other set $\mathcal U'$ (possibly with a different value of $d$). Given any set of vertices of $G_{\mathcal U}$ it is easy to find a set of vertices in one of the components with the same size and edge boundary, so the condition that $\mathcal U$ generates $\mathbb Z^d$ as a group is not a serious restriction.

\begin{figure}[t]
\begin{center}
\begin{tikzpicture}%
	[x={(0.043036cm, -0.948593cm)},
	y={(0.378649cm, 0.305923cm)},
	z={(0.924539cm, -0.081136cm)},
	scale=0.300000,
	edge/.style={color=black, thick},
	axis/.style={color=black, thick, dotted}]
\draw[axis] (0,-9,0) -- (0,-20,0);
\draw[axis] (0, 0,9) -- (0, 0,13);
\draw[axis] (-9, 0, 0) -- (-13, 0, 0);
\draw[edge] (-9.00000, -3.00000, -1.00000) -- (-9.00000, -3.00000, 1.00000);
\draw[edge] (-9.00000, -3.00000, -1.00000) -- (-9.00000, -1.00000, -3.00000);
\draw[edge] (-9.00000, -3.00000, -1.00000) -- (-7.00000, -5.00000, -3.00000);
\draw[edge] (-9.00000, -3.00000, 1.00000) -- (-9.00000, -1.00000, 3.00000);
\draw[edge] (-9.00000, -3.00000, 1.00000) -- (-7.00000, -5.00000, 3.00000);
\draw[edge] (-9.00000, -1.00000, -3.00000) -- (-9.00000, 1.00000, -3.00000);
\draw[edge] (-9.00000, -1.00000, -3.00000) -- (-7.00000, -3.00000, -5.00000);
\draw[edge] (-9.00000, -1.00000, 3.00000) -- (-9.00000, 1.00000, 3.00000);
\draw[edge] (-9.00000, -1.00000, 3.00000) -- (-7.00000, -3.00000, 5.00000);
\draw[edge] (-9.00000, 1.00000, -3.00000) -- (-9.00000, 3.00000, -1.00000);
\draw[edge] (-9.00000, 1.00000, 3.00000) -- (-9.00000, 3.00000, 1.00000);
\draw[edge] (-9.00000, 1.00000, 3.00000) -- (-7.00000, 3.00000, 5.00000);
\draw[edge] (-9.00000, 3.00000, -1.00000) -- (-9.00000, 3.00000, 1.00000);
\draw[edge] (-9.00000, 3.00000, 1.00000) -- (-7.00000, 5.00000, 3.00000);
\draw[edge] (9.00000, 1.00000, 3.00000) -- (9.00000, -1.00000, 3.00000);
\draw[edge] (9.00000, 1.00000, 3.00000) -- (7.00000, 3.00000, 5.00000);
\draw[edge] (9.00000, -1.00000, 3.00000) -- (9.00000, -3.00000, 1.00000);
\draw[edge] (9.00000, -1.00000, 3.00000) -- (7.00000, -3.00000, 5.00000);
\draw[edge] (9.00000, -3.00000, 1.00000) -- (9.00000, -3.00000, -1.00000);
\draw[edge] (9.00000, -3.00000, 1.00000) -- (7.00000, -5.00000, 3.00000);
\draw[edge] (9.00000, -3.00000, -1.00000) -- (7.00000, -5.00000, -3.00000);
\draw[edge] (-7.00000, -5.00000, -3.00000) -- (-7.00000, -3.00000, -5.00000);
\draw[edge] (-7.00000, -5.00000, -3.00000) -- (-5.00000, -7.00000, -3.00000);
\draw[edge] (-7.00000, -5.00000, 3.00000) -- (-7.00000, -3.00000, 5.00000);
\draw[edge] (-7.00000, -5.00000, 3.00000) -- (-5.00000, -7.00000, 3.00000);
\draw[edge] (-7.00000, -3.00000, -5.00000) -- (-5.00000, -3.00000, -7.00000);
\draw[edge] (-7.00000, -3.00000, 5.00000) -- (-5.00000, -3.00000, 7.00000);
\draw[edge] (-7.00000, 3.00000, 5.00000) -- (-7.00000, 5.00000, 3.00000);
\draw[edge] (-7.00000, 3.00000, 5.00000) -- (-5.00000, 3.00000, 7.00000);
\draw[edge] (7.00000, 3.00000, 5.00000) -- (5.00000, 3.00000, 7.00000);
\draw[edge] (7.00000, -3.00000, 5.00000) -- (7.00000, -5.00000, 3.00000);
\draw[edge] (7.00000, -3.00000, 5.00000) -- (5.00000, -3.00000, 7.00000);
\draw[edge] (7.00000, -3.00000, -5.00000) -- (7.00000, -5.00000, -3.00000);
\draw[edge] (7.00000, -3.00000, -5.00000) -- (5.00000, -3.00000, -7.00000);
\draw[edge] (7.00000, -5.00000, 3.00000) -- (5.00000, -7.00000, 3.00000);
\draw[edge] (7.00000, -5.00000, -3.00000) -- (5.00000, -7.00000, -3.00000);
\draw[edge] (-5.00000, -7.00000, -3.00000) -- (-3.00000, -9.00000, -1.00000);
\draw[edge] (-5.00000, -7.00000, -3.00000) -- (-3.00000, -7.00000, -5.00000);
\draw[edge] (-5.00000, -7.00000, 3.00000) -- (-3.00000, -9.00000, 1.00000);
\draw[edge] (-5.00000, -7.00000, 3.00000) -- (-3.00000, -7.00000, 5.00000);
\draw[edge] (-5.00000, -3.00000, -7.00000) -- (-3.00000, -5.00000, -7.00000);
\draw[edge] (-5.00000, -3.00000, -7.00000) -- (-3.00000, -1.00000, -9.00000);
\draw[edge] (-5.00000, -3.00000, 7.00000) -- (-3.00000, -5.00000, 7.00000);
\draw[edge] (-5.00000, -3.00000, 7.00000) -- (-3.00000, -1.00000, 9.00000);
\draw[edge] (-5.00000, 3.00000, 7.00000) -- (-3.00000, 1.00000, 9.00000);
\draw[edge] (-5.00000, 3.00000, 7.00000) -- (-3.00000, 5.00000, 7.00000);
\draw[edge] (5.00000, 3.00000, 7.00000) -- (3.00000, 1.00000, 9.00000);
\draw[edge] (5.00000, -3.00000, 7.00000) -- (3.00000, -1.00000, 9.00000);
\draw[edge] (5.00000, -3.00000, 7.00000) -- (3.00000, -5.00000, 7.00000);
\draw[edge] (5.00000, -3.00000, -7.00000) -- (3.00000, -5.00000, -7.00000);
\draw[edge] (5.00000, -7.00000, 3.00000) -- (3.00000, -7.00000, 5.00000);
\draw[edge] (5.00000, -7.00000, 3.00000) -- (3.00000, -9.00000, 1.00000);
\draw[edge] (5.00000, -7.00000, -3.00000) -- (3.00000, -7.00000, -5.00000);
\draw[edge] (5.00000, -7.00000, -3.00000) -- (3.00000, -9.00000, -1.00000);
\draw[edge] (-3.00000, -9.00000, -1.00000) -- (-3.00000, -9.00000, 1.00000);
\draw[edge] (-3.00000, -9.00000, -1.00000) -- (-1.00000, -9.00000, -3.00000);
\draw[edge] (-3.00000, -9.00000, 1.00000) -- (-1.00000, -9.00000, 3.00000);
\draw[edge] (-3.00000, -7.00000, -5.00000) -- (-3.00000, -5.00000, -7.00000);
\draw[edge] (-3.00000, -7.00000, -5.00000) -- (-1.00000, -9.00000, -3.00000);
\draw[edge] (-3.00000, -7.00000, 5.00000) -- (-3.00000, -5.00000, 7.00000);
\draw[edge] (-3.00000, -7.00000, 5.00000) -- (-1.00000, -9.00000, 3.00000);
\draw[edge] (-3.00000, -5.00000, -7.00000) -- (-1.00000, -3.00000, -9.00000);
\draw[edge] (-3.00000, -5.00000, 7.00000) -- (-1.00000, -3.00000, 9.00000);
\draw[edge] (-3.00000, -1.00000, -9.00000) -- (-1.00000, -3.00000, -9.00000);
\draw[edge] (-3.00000, -1.00000, 9.00000) -- (-3.00000, 1.00000, 9.00000);
\draw[edge] (-3.00000, -1.00000, 9.00000) -- (-1.00000, -3.00000, 9.00000);
\draw[edge] (-3.00000, 1.00000, 9.00000) -- (-1.00000, 3.00000, 9.00000);
\draw[edge] (-3.00000, 5.00000, 7.00000) -- (-1.00000, 3.00000, 9.00000);
\draw[edge] (3.00000, 1.00000, 9.00000) -- (3.00000, -1.00000, 9.00000);
\draw[edge] (3.00000, 1.00000, 9.00000) -- (1.00000, 3.00000, 9.00000);
\draw[edge] (3.00000, -1.00000, 9.00000) -- (1.00000, -3.00000, 9.00000);
\draw[edge] (3.00000, -5.00000, 7.00000) -- (3.00000, -7.00000, 5.00000);
\draw[edge] (3.00000, -5.00000, 7.00000) -- (1.00000, -3.00000, 9.00000);
\draw[edge] (3.00000, -5.00000, -7.00000) -- (3.00000, -7.00000, -5.00000);
\draw[edge] (3.00000, -5.00000, -7.00000) -- (1.00000, -3.00000, -9.00000);
\draw[edge] (3.00000, -7.00000, 5.00000) -- (1.00000, -9.00000, 3.00000);
\draw[edge] (3.00000, -7.00000, -5.00000) -- (1.00000, -9.00000, -3.00000);
\draw[edge] (-1.00000, -9.00000, -3.00000) -- (1.00000, -9.00000, -3.00000);
\draw[edge] (-1.00000, -9.00000, 3.00000) -- (1.00000, -9.00000, 3.00000);
\draw[edge] (-1.00000, -3.00000, -9.00000) -- (1.00000, -3.00000, -9.00000);
\draw[edge] (-1.00000, -3.00000, 9.00000) -- (1.00000, -3.00000, 9.00000);
\draw[edge] (-1.00000, 3.00000, 9.00000) -- (1.00000, 3.00000, 9.00000);
\draw[edge] (3.00000, -9.00000, 1.00000) -- (3.00000, -9.00000, -1.00000);
\draw[edge] (3.00000, -9.00000, 1.00000) -- (1.00000, -9.00000, 3.00000);
\draw[edge] (3.00000, -9.00000, -1.00000) -- (1.00000, -9.00000, -3.00000);
\end{tikzpicture}
\caption{A near-optimal shape for the edge isoperimetric problem in $(\mathbb Z^3, l_\infty)$.  The dotted lines represent the coordinate axes.}
\label{fig:3dshape}
\end{center}
\end{figure}
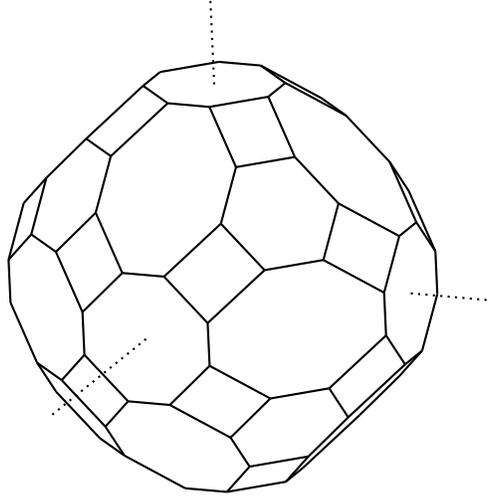

When $\mathcal U = \{(\pm1, 0, \ldots, 0), \ldots, (0, \ldots, 0, \pm 1)\}$, we have $Z = [-1, 1]^d$, so we recover an asymptotic version of the edge isoperimetric theorem for $(\mathbb Z^d, l_1)$.
When $\mathcal U = \{-1, 0, 1\}^d \setminus \{(0,\ldots,0)\}$, corresponding to $(\mathbb Z^d, l_\infty)$, the zonotope $Z$ is more complicated.
For $d=2$, it is an octagon obtained by cutting the corners off a square through points one third of the way along each side.
In this case there is in fact a nested sequence of optimal sets, which interpolate between discrete versions of this octagon.
This was proved by Brass \cite[Theorem~3]{B96} as part of his work on the Erd\H{o}s distance problem in $(\mathbb R^2, l_\infty)$.
The zonotope $Z$ for $d=3$ is shown in Figure~\ref{fig:3dshape}.
When $\mathcal U = \{\pm (1, 0), \pm (0, 1), \pm(1,1)\}$, $G_{\mathcal U}$ is isomorphic to the triangular lattice and $Z$ corresponds to a regular hexagon.
Harper \cite[Theorem~7.2]{H04} proved that for this lattice there is an optimal ordering interpolating between regular hexagons.
As far as we are aware, these are the only cases of Theorem~\ref{t:main} appearing in the literature.
We might also ask about the vertex isoperimetric problem on these more general lattices.
It turns out that this question has already been answered by Ruzsa \cite{R1995}.

\begin{theorem}[\cite{R1995}] \label{t:ruzsa}
Let $\mathcal U = \{u_1, \ldots, u_k\}$ be a finite set of non-zero vectors that generate $\mathbb Z^d$ as a group.
Let $U$ be the convex hull of $\mathcal U \cup \{0\}$ in $\mathbb R^d$.
For every $\delta > 0$, there is an $n_0 = n_0(\delta, \mathcal U)$ such that, for every $n \geq n_0$, 
\[
(1-\delta) d \vol(U)^{1/d} n^{1-1/d} \leq \partial_{v,G_{\mathcal U}}(n) \leq (1+\delta) d \vol(U)^{1/d} n^{1-1/d}.
\]
\end{theorem}

Ruzsa presents his result in the language of sumsets: for a fixed subset $B$ of $\mathbb Z^d$ such that $B - B$ generates $\mathbb Z^d$ as a group, he seeks to minimise $|S + B|$ over all subsets $S$ of $\mathbb Z^d$ of a given size.
The vertex boundary of a set $S$ in $G_{\mathcal U}$ can be expressed as
\[
\vb(S) = |S + (\mathcal U \cup \{0\})| - |S|,
\]
so, after translating $B$ so that it contains $0$, the two problems are easily seen to be equivalent.

The structures of the optimal sets in Theorems~\ref{t:main} and~\ref{t:ruzsa} do not seem promising for the use of compression techniques. To use compressions to prove that the shape $Z$ in Figure~\ref{fig:3dshape} is optimal for the edge isoperimetric problem in $(\bb{Z}^3,l_{\infty})$ we would like to take slices of our graph isomorphic to $(\bb{Z}^2,l_{\infty})$ and show that the intersection of each slice with $Z$ is itself optimal; 
but most cross-sections through $Z$ are not octagons of the correct shape.
To prove Theorem~\ref{t:ruzsa} Ruzsa instead solved a continuous approximation, then used combinatorial methods to show that the approximation was good.
We give a more detailed sketch in Section~\ref{s:continuous}.

We would like to take the same approach to proving Theorem~\ref{t:main}. The edge isoperimetric problem has a natural continuous analogue, and the solution suggests the correct statement of Theorem~\ref{t:main}. (Indeed, Theorem~\ref{t:main} has been conjectured independently by Tsukerman and Veomett \cite{TV2017}.) However, it is not clear that the continuous analogue is a good approximation to the original discrete problem. Instead, we will show that the edge isoperimetric problem can be related to the vertex isoperimetric problem in a different lattice.

In Section~\ref{s:continuous} we sketch the relationship between isoperimetric problems in $G_{\mathcal U}$ and their continuous analogues. In Section~\ref{s:proof} we prove Theorem~\ref{t:main}. Finally, in Section~\ref{s:open-problems} we discuss some open problems.

\section{Relation to continuous problems} \label{s:continuous}

In this section we indicate connections between isoperimetric problems in Cayley graphs on $\mathbb Z^d$ and classical results from convex geometry. For more background on the geometric results mentioned below, see for example the book of Schneider \cite{S13}.

\subsection{Vertex isoperimetry} \label{s:continuous-vertex}

Write $\mathcal U_0 = \mathcal U \cup \{0\}$ and recall that the vertex isoperimetric problem for $G_{\mathcal U}$ is equivalent to minimising $|S + \mathcal U_0|$ over all subsets $S$ of $\mathbb Z^d$ of a given size.
Given a subset $S$ of $\mathbb Z^d$, write $\bar S = S + [0,1]^d$ for the subset of $\mathbb R^d$  obtained by replacing each point of $S$ by a unit cube.
We hope that $\bar S$ is a good continuous approximation to $S$; for example, $|S| = \vol(\bar S)$.
We seek to bound $|S + \mathcal U_0|$ from below using the Brunn--Minkowski theorem.

\begin{theorem}[Brunn--Minkowski]
Let $A$ and $B$ be compact subsets of $\mathbb R^d$ with $\vol(A), \vol(B) > 0$.
Then $\vol(A + B)^{1/d} \geq \vol(A)^{1/d} + \vol(B)^{1/d}$, with equality if and only if $A$ and $B$ are convex and homothetic (that is, equal up to scaling and translation).
\end{theorem}

Heuristically we have
\begin{align*}
|S + \mathcal U_0| & \approx \vol(\bar S + \bar{\mathcal U}_0)
\\               & \geq (\vol(\bar S)^{1/d} + \vol(\bar{\mathcal U}_0)^{1/d})^d
\\               & \approx \vol(\bar S) + d \vol(\bar{\mathcal U}_0)^{1/d} \vol(\bar S)^{(d-1)/d}
\\               & = |S| + d \vol(\bar{\mathcal U}_0)^{1/d} |S|^{(d-1)/d},
\end{align*}
when $S$ is large. The main problem with this argument is that the inequality from the Brunn--Minkowski theorem will be weak if $\bar {\mathcal U}_0$ is far from a convex set.
To fix this, Ruzsa instead considers $|S + t\mathcal U_0|$ for some large, fixed $t$.
Since $\mathcal U_0$ generates $\mathbb Z^d$ as a group, for large $t$ the sumset $t\mathcal U_0$ `fills space' and is well approximated by its convex hull. Ruzsa completes his proof by using Pl\"unnecke's inequality \cite{P1970}, a result from additive combinatorics, to relate the size of $S + t\mathcal U_0$ to that of $S + \mathcal U_0$.

\subsection{Edge isoperimetry} \label{s:continuous-edge}

The first step to solving the vertex isoperimetric problem was to rephrase the problem in terms of sumsets.
We can attempt the same process here.
Fix a dimension $d$ and a set $\mathcal U$ as in the statement of Theorem~\ref{t:main}.
For $i \in [k]$ and $S \subset \mathbb Z^d$, let $S_i = S + \{0, u_i\}$ be the set obtained by `pushing out' in the $u_i$-direction. Then $\eb_i(S) = |S_i| - |S|$ is the contribution to the edge boundary made by edges leaving $S$ in the $u_i$-direction.
The continuous analogue is the `weighted surface area' of $\bar S$ in direction $u_i$; that is, the infinitesimal change in volume when we take an infinitesimal step in the $u_i$ direction.
More precisely, write $\lineseg u i$ for the line segment from $0$ to $u_i$ in $\mathbb R^d$.
Then we are interested in the quantity
\[
\lim_{\epsilon \to 0^+} \frac{\vol(\bar S + \epsilon \cdot \lineseg u i) - \vol(\bar S)} \epsilon.
\]
This is an instance of a geometric quantity known as a `mixed volume'.
Write $\mathcal K_d$ for the set of compact, convex, non-empty subsets of $\mathbb R^d$.
There is a unique function $V : \mathcal K_d^d \to \mathbb R_{\geq 0}$ with the following properties.
\begin{itemize}
\item $V$ is symmetric under permutations of its arguments.
\item $V(K, \ldots, K) = \vol(K)$.
\item $V$ is `linear' in each coordinate: 
\[
V(\alpha_1 \cdot L_1 + \alpha_2 \cdot L_2, K_2, \ldots, K_d) = \alpha_1 V(L_1, K_2, \ldots, K_d) + \alpha_2 V(L_2, K_2, \ldots, K_d).
\]
\end{itemize}
We call $V(K_1, \ldots, K_d)$ the \defn{mixed volume} of $K_1, \ldots, K_d$.

\begin{remark}[Zonotopes] \label{zonotopes}
The `linearity' property of the mixed volume leads to particular interest in convex bodies that can be expressed as Minkowski sums of simpler convex bodies.
The simplest non-trivial convex bodies are line segments; a \defn{zonotope} is any convex body that can be expressed as a Minkowski sum of line segments.
The shape $Z$ appearing in the statement of Theorem~\ref{t:main} is a zonotope, as the convex hull of $Z_0 = \{0, u_1\} + \cdots + \{0, u_k\}$ is the set
\[
\left\{\sum_{i=1}^k \lambda_i u_i \,:\, 0 \leq \lambda_i \leq 1 \text{ for each } i\right\} = \sum_{i=1}^k \lineseg u i.
\]
Examples of zonotopes include cubes (in any dimension) and centrally symmetric convex polygons in $\mathbb R^2$.
\end{remark}

\begin{theorem}[Minkowski]
Let $K_1, \ldots, K_r$ be compact, convex, non-empty subsets of $\mathbb R^d$.
Then for all $\lambda_1, \ldots, \lambda_r \in \mathbb R_{\geq 0}$,
\[
\vol(\lambda_1 \cdot K_1 + \cdots + \lambda_r \cdot K_r) = \sum_{\mathbf j \in [r]^d} V(K_{j_1}, \ldots, K_{j_d}) \lambda_{j_1} \cdots \lambda_{j_d}.
\]
\end{theorem}

For $A, B \in \mathcal K_d$,
\[
\vol(A + \epsilon \cdot B) = V(A, \ldots, A) + d \epsilon V(B, A, \ldots, A) + O(\epsilon^2),
\]
so 
\[
d V(B, A, \ldots, A) = \lim_{\epsilon \to 0^+} \frac{\vol(A + \epsilon\cdot B) - \vol(A)} \epsilon
\] is a measure of how fast $A$ grows when we add a small copy of $B$ to it. When $B$ is a Euclidean ball, $dV(B, A, \ldots, A)$ is the usual `surface area' (codimension $1$ volume of the boundary) of $A$.
By varying $B$ we obtain different notions of surface area.
The following result is closely related to the Brunn--Minkowski theorem.

\begin{theorem}[Minkowski's first inequality]
Let $A, B \in \mathcal K_d$. Then
\[
V(B, A, \ldots, A) \geq \vol(B)^{1/d}\vol(A)^{(d-1)/d},
\]
with equality if and only if $A$ and $B$ are homothetic.
\end{theorem}

Minkowski's first inequality is an extremely powerful tool for proving isoperimetric theorems.
For example, the classical isoperimetric theorem follows immediately by taking $B$ to be a Euclidean ball.
For our edge isoperimetric problem we have the following heuristic argument.
\begin{align*}
\eb(S) & = \sum_{i=1}^k \eb_i(S) \approx \sum_{i=1}^k dV(\lineseg u i, \bar S, \ldots, \bar S)
\\     & = d V\left(\sum_{i=1}^k \lineseg u i, \bar S, \ldots, \bar S\right)
         = d V\left(Z, \bar S, \ldots, \bar S\right)
\\     & \geq d \vol(Z)^{1/d}\vol(\bar S)^{(d-1)/d}
         = d \vol(Z)^{1/d}|S|^{(d-1)/d},
\end{align*}
which should be close to optimal when $\bar S$ is close to a scaling and translation of $Z$.
This strongly suggests the correct statement of Theorem~\ref{t:main}.
However, making this argument precise appears not to be straightforward.
The essential problem is that the volume of the projection of a unit cube varies between $1$ and $\sqrt d$, and so the approximation of the discrete problem by the continuous problem is not automatically good in every case.
We shall instead take a slightly different approach.
The philosophy is largely the same, but we will do as much of our approximation as possible on the discrete side before applying Theorem~\ref{t:ruzsa} to handle the transfer from the discrete to the continuous problem implicitly.

\section{Proof of Theorem~\ref{t:main}}\label{s:proof}

Throughout this section we fix a dimension $d$ and a set $\mathcal U$ as in the statement of Theorem~\ref{t:main}.
All constants may depend on $d$ and $\mathcal U$ but are otherwise absolute. 
Recall that we write $S_i = S + \{0, u_i\}$, and $\eb_i(S) = |S_i| - |S|$ for the contribution to $\eb(S)$ in the $u_i$-direction.
For $\hist x \in [k]^t$, we write $\push S x = S_{x_1\cdots x_t} = S + \{0, u_{x_1}\} + \cdots +\{0, u_{x_t}\}$. 
Note that $\push S x$ depends only on the number of coordinates of $\hist x$ taking each value, not on their order.
We call a finite subset $S$ of $\mathbb Z^d$ \defn{$\epsilon$-close to optimal} if $\eb(S) \leq (1 + \epsilon)\ebopt(|S|)$.

\subsection{Lower bound}\label{s:lower}

We first collect together some simple properties of $\ebopt$.

\begin{lemma} \label{basics}
\begin{itemize}
\item[\rm(i)] $\ebopt(n)$ is an increasing function of $n$.
\item[\rm(ii)] There are constants $0 < c < C$ such that $cn^{1-1/d} \leq \ebopt(n) \leq Cn^{1-1/d}$.
\item[\rm(iii)] For all $m, n \in \mathbb N$, $\ebopt(n+m) - \ebopt(n) \leq Cm^{1-1/d}$.
\end{itemize}
\end{lemma}

\begin{proof}
(i) Choose $S \subset \mathbb Z^d$ with $|S| = n+1$ and $\eb(S) = \ebopt(n+1)$.
Let $w$ be a vector in $\mathbb R^d$ with algebraically independent entries;
then $ v\mapsto v\cdot w$ is an injection from $\mathbb Z^d$ to $\mathbb R$.
Let $v$ be the unique element of $S$ with $v \cdot w$ maximal. 
Then
\begin{itemize}
\item for each $i \in [k]$ such that $u_i\cdot w > 0$, $v+u_i \not\in S$.
\item for each $i \in [k]$ such that $v - u_i \in S$, $u_i \cdot w > 0$.
\end{itemize}
Hence
\begin{align*}
\ebopt_{\mathcal U}(n) & \leq \eb_{\mathcal U}(S\setminus\{v\}) 
\\ & = \eb_{\mathcal U}(S) - |\{i \in [k] \, : \, v + u_i \notin S\}| + |\{i \in [k] \, : \, v - u_i \in S)\}|
\\ & \leq \eb_{\mathcal U}(S) - |\{i \in [k] \, : \, u_i \cdot w > 0\}| + |\{i \in [k] \, : \, u_i \cdot w > 0)\}|
\\ &= \eb_{\mathcal U}(S) = \ebopt_{\mathcal U}(n+1).
\end{align*}

(ii) The maximum in-degree of $G_{\mathcal U}$ is $k$, so this follows from Theorem~\ref{t:ruzsa} and the inequalities \eqref{differ-by-a-constant}.

(iii) Let $S$ be a set of $n$ points with $\eb(S) = \ebopt(n)$ and let $T$ be a set of $m$ points with $\eb(T) = \ebopt(m)$.  
By translating if necessary we may assume that the distance between $S$ and $T$ is greater than the length of the longest vector in $\mathcal U$.
Then 
\[
\ebopt(m+n) \leq  \eb(S \cup T) = \ebopt(n) + \ebopt(m) \leq \ebopt(n) + C m^{1-1/d},
\]
by (ii).
\end{proof}

Lemma~\ref{basics}(iii) tells us that $\ebopt(n)$ varies only slowly with $n$.
It follows that if we make a small modification to an optimal set then it should remain close to optimal.
We now prove that the `pushing out' operation $S \mapsto S_i$ causes suitably small modifications when applied to sets that are already close to optimal.

\begin{lemma} \label{pushing-preserves-optimality}
For every $\epsilon > 0$, there is an $n_0 = n_0(\epsilon)$ such that, if $S$ is $\epsilon$-close to optimal and $|S| \geq n_0$, then $S_i$ is $3\epsilon$-close to optimal for each $i \in [k]$.
\end{lemma}

\begin{proof}
For $i \in [k]$, let $F_i = \{v \in S \, : \, v + u_i \notin S\}$ be the \defn{frontier} of $S$ in the $u_i$-direction.
We claim that, for each $j \in [k]$,
\begin{align}
\eb_j(S_i) - \eb_j(S) \leq \eb_j(S) - \eb_j(S \setminus F_i). \label{claim1}
\end{align}
That is, if the boundary in the $u_j$-direction gets much larger when we pass from $S$ to $S_i$, then it also gets much smaller if we pass from $S$ to $S \setminus F_i$.
This seems unlikely if $S$ is close to optimal.

If $T, T'$ are disjoint finite subsets of $\mathbb Z^d$, then
\begin{align}\label{e:split}
\begin{split}
\eb_j(T \cup T') - \eb_j(T) & = |\{v \in T' \, : \, v + u_j \notin T \cup T'\}| - |\{v \in T' \, : \, v - u_j \in T\}|
\\ & = |T' \setminus \big( (T \cup T') - u_j \big)| - |T' \cap (T + u_j)|,
\end{split}
\end{align}
as every edge that changes from contributing to the edge boundary to not contributing or vice versa has exactly one endpoint in the set of new vertices $T'$.

Taking $T = S$ and $T' = S_i \setminus S = F_i + u_i$ gives
\begin{align*}
\eb_j(S_i) - \eb_j(S)     & = |(F_i+u_i) \setminus (S_i - u_j)| - |(F_i+u_i) \cap (S + u_j)| 
\\ & = |(F_i + u_j) \setminus (S_i - u_i)| - |(F_i - u_j) \cap (S - u_i)|,
\end{align*}
and taking $T = S \setminus F_i$ and $T' = F_i$ gives
\begin{align*}
\eb_j(S) - \eb_j(S \setminus F_i) & = |F_i \setminus (S - u_j)| - |F_i \cap ((S \setminus F_i) + u_j)|
\\ & = |(F_i + u_j) \setminus S| - |(F_i - u_j) \cap (S \setminus F_i)|.
\end{align*}
Observing that $S \subseteq S_i - u_i$ and $S \setminus F_i \subseteq S- u_i$ proves \eqref{claim1}.

Now  
\[
|F_i| = \eb_i(S) \leq \eb(S) \leq (1+\epsilon) \ebopt(|S|) \leq (1+\epsilon)C |S|^{1-1/d},
\] 
hence by summing \eqref{claim1} over $j \in [k]$ and applying Lemma \ref{basics},
\begin{align*}
\eb(S_i) & \leq 2 \eb(S) - \eb(S \setminus F_i)
\\ & \leq (2 + 2 \epsilon)\ebopt(|S|) - \ebopt(|S \setminus F_i|)
\\ & =    (1 + 2 \epsilon)\ebopt(|S|) + \ebopt(|S|)- \ebopt(|S \setminus F_i|)
\\ & \leq (1 + 2 \epsilon)\ebopt(|S|) + C |F_i|^{1-1/d}
\\ & \leq (1 + 2 \epsilon)\ebopt(|S|) + C\big((1+\epsilon)C |S|^{(1-1/d)}\big)^{1-1/d}
\\ & \leq (1 + 3 \epsilon)\ebopt(|S|) 
 \leq (1 + 3 \epsilon)\ebopt(|S_i|),
\end{align*}
provided $|S|$ is sufficiently large, depending on $\epsilon$.
\end{proof}

Repeated application of Lemma~\ref{pushing-preserves-optimality} tells us that, if $S$ is $\epsilon$-close to optimal, and $\hist x \in [k]^t$ with $t$ small, then $\eb(\push S x)$ is not much bigger than $\ebopt(|\push S x|)$. 
Next we show that, in fact, $\eb(\push S x)$ is not much bigger than $\ebopt(|S|)$.

\begin{lemma} \label{Sx}
There exist constants $(C_t)_{t=0}^\infty$ such that the following holds. For every $t$ and every $0 < \epsilon < 1$, there is an $n_0 = n_0(\epsilon, t)$ such that, if $S$ is $\epsilon$-close to optimal, $|S| \geq n_0$ and $\hist x \in [k]^t$, then
\begin{itemize}
\item[\rm(i)]  $|\push S x| - |S| \leq C_t |S|^{1-1/d}$.
\item[\rm(ii)] $\eb(\push S x) \leq (1+2\cdot 3^t \epsilon) \ebopt (|S|)$.
\end{itemize}
\end{lemma}

\begin{proof}
(i) We proceed by induction on $t$. We can take $C_0 = 0$.
So assume we have found suitable $C_0, \ldots, C_t$, let $\hist x \in [k]^t$ and let $i \in [k]$. 
Note that, by Lemma \ref{pushing-preserves-optimality}, since $S$ is $\epsilon$-close to optimal, $\push S x$ is $3^{t}\epsilon$-close to optimal. 
Hence, by Lemma \ref{basics},
\begin{align*}
|S_{\hist x i}| - |\push S x| & \leq \eb(\push S x) \leq (1+3^t\epsilon) \ebopt(|\push S x|)
\\ & = (1 + 3^t \epsilon) \big(\ebopt(|S|) + \ebopt(|\push S x|) - \ebopt(|S|)\big)
\\ & \leq (1 + 3^t \epsilon) \big(C|S|^{1-1/d} + C(|\push S x|-|S|)^{1-1/d}\big).
\end{align*}
Also, by induction, $|\push S x| - |S| \leq C_t |S|^{1-1/d}$, and so
\begin{align*}
|S_{\hist x i}| - |S| & = |S_{\hist x i}| - |\push S x| + |\push S x| - |S| 
\\ & \leq (1 + 3^t \epsilon) \big(C|S|^{1-1/d} + C(|\push S x|-|S|)^{1-1/d}\big) + |\push S x| - |S|,
\\ & \leq (1 + 3^t \epsilon) \big(C|S|^{1-1/d} + C(C_t|S|^{1-1/d})^{1-1/d}\big) + C_t |S|^{1-1/d} \\ & \leq C_{t+1}|S|^{1-1/d},
\end{align*}
for some constant $C_{t+1}$.

(ii) By Lemma \ref{basics}, and part (i)
\begin{align*}
\ebopt(|\push S x|) - \ebopt(|S|) & \leq C(|\push S x| - |S|)^{1-1/d} \leq C(C_t|S|^{1-1/d})^{1-1/d},
\end{align*}
and so, since $\push S x$ is $3^{t}\epsilon$-close to optimal, it follows that
\begin{align*}
\eb(\push S x) & \leq (1+3^{t}\epsilon) \ebopt(|\push S x|)
= (1+3^{t}\epsilon) \big( \ebopt(|S|) + \ebopt(|\push S x|) - \ebopt(|S|) \big)
\\ & \leq (1+3^{t}\epsilon) \big( \ebopt(|S|) + C(C_t|S|^{1-1/d})^{1-1/d} \big)
\\ & \leq (1+2\cdot 3^{t} \epsilon) \ebopt (|S|),
\end{align*}
with the final inequality following from Lemma~\ref{basics}(ii), provided $|S|$ is sufficiently large, depending on $\epsilon$ and $t$.
\end{proof}

To complete the proof of Theorem~\ref{t:main}, we would like to argue as follows.
Let $S$ be close to optimal.
Then,
\begin{align*}
\ebopt(|S|) & \approx \eb(S) = \sum_{i=1}^k (|S_i| - |S|) 
\\ & \approx (|S_1| - |S|) + (|S_{12}| - |S_1|) + \cdots + (|S_{1\ldots k}| - |S_{1 \ldots (k-1)}|) \numberthis \label{telescope}
\\ & = |S_{1\ldots k}| - |S| = |S + Z_0| - |S|.
\end{align*}
The final expression is a vertex boundary, so can be bounded below using Theorem~\ref{t:ruzsa}.
However, there is a problem with the second approximation step.
We know that, say, $\eb(S_1) \approx \eb(S)$.
But for the approximation to hold term by term we would need the stronger result that $\eb_2(S_1) \approx \eb_2(S)$.
To fix this we note that these approximations cannot all be overestimates, and it is at least plausible that if we consider much longer telescoping sums obtained by extending in a random direction at each stage then we will obtain a more accurate approximation to a related quantity.
This is the idea behind the proof below.

\begin{proof}[Proof of Theorem~\ref{t:main}, lower bound]
First fix $\eta > 0$ small depending on $\delta$, $t \in \mathbb N$ large depending on $\eta$, $\epsilon > 0$ small depending on $t$, $n_0$ sufficiently large depending on $\epsilon$, then let $S$ be $\epsilon$-close to optimal with $|S| \geq n_0$.
By applying Lemma~\ref{Sx}(ii) to each $\push S x$,
\begin{align*}
t (1 + 2 \cdot 3^{t-1} \epsilon) \ebopt(|S|) & \geq \sum_{s=0}^{t-1} \E_{\hist x \in [k]^s} \eb(\push S x)
 = \sum_{s=0}^{t-1} \E_{\hist x \in [k]^s} {\textstyle \sum_{i=1}^k} (|S_{\hist x i}| - |\push S x|)
\\ & = k \sum_{s=0}^{t-1} \Big( \E_{\hist x \in [k]^{s+1}} |\push S x| - \E_{\hist x \in [k]^s} |\push S x| \Big)
\\ & = k \E_{\hist x \in [k]^t} (|\push S x| - |S|),
\end{align*}
where each expectation is over a uniform choice of $\hist x$.

We now observe that a random $\hist x \in [k]^t$ is very likely to have about $t/k$ entries of each value.
Since $\push S x$ only depends on the number of coordinates of $\hist x$ with each value, not on their order, all of the terms in the final expectation are then very close to one single value.

To make this precise, observe that, with probability at least $1 - k e^{-2\eta^2t/k^2}$, for each $i \in [k]$, at least $(1-\eta)t/k$ entries of $\hist x$ take the value $i$.
(The number of coordinates taking value $i$ has binomial distribution $B(t,1/k)$, so this is a simple application of Chernoff's inequality (see \cite[Remark~2.5]{JLR}) and the union bound.)

For all such $\hist x$, $\push S x \supseteq \push S y$, where $\hist y$ is a vector of length $k\lceil (1-\eta)t/k \rceil$ with coordinates taking each value in $[k]$ equally often.
Now
\[
\push S y = S + \lceil (1-\eta)t/k \rceil (\{0, u_1\} + \cdots + \{0, u_k\}) = S + \lceil (1-\eta)t/k \rceil Z_0,
\]
and so by the Chernoff bound
\[
\E_{\hist x \in [k]^t} (|\push S x| - |S|) \geq (1 - k e^{-2\eta^2t/k^2})(|S + \lceil(1-\eta)t/k\rceil Z_0| - |S|). 
\]
Note that the final term is the vertex boundary of the set $S$ in the graph $G_{\mathcal{U}'}$ with $\mathcal{U}' = \lceil(1-\eta)t/k\rceil Z_0 \setminus\{0\}$. 

Since $\mathcal{U}' \supseteq \mathcal{U}$ it generates $\mathbb{Z}^d$ as a group, so we may apply Theorem~\ref{t:ruzsa} to $\mathcal{U}'$ with $\delta = \eta$ to obtain
\begin{align*}
\E_{\hist x \in [k]^t} (|\push S x| - |S|) & \geq (1 - k e^{-2\eta^2t/k^2})(1-\eta)d\vol(\lceil(1-\eta)t/k\rceil Z)^{1/d}|S|^{1-1/d}
\\ & \geq (1 - k e^{-2\eta^2t/k^2})(1-\eta)^2 \frac t k d\vol(Z)^{1/d}|S|^{1-1/d}.
\end{align*}
Hence
\begin{align*}
\ebopt(|S|) & \geq \frac {(1 - k e^{-2\eta^2t/k^2})(1-\eta)^2} {(1 + 2 \cdot 3^{t-1} \epsilon)} d \vol(Z)^{1/d}|S|^{1-1/d}
\\ & \geq (1-\delta)d \vol(Z)^{1/d}|S|^{1-1/d}. \qedhere
\end{align*}

\end{proof}

\subsection{Upper bound} \label{s:upper}

In this section we show that the lower bound on $\ebopt(n)$ from Section~\ref{s:lower} is sharp by exhibiting a family of sets for which we can prove a matching upper bound on the edge boundary. We again seek to use the approximation~\eqref{telescope}.
Since we only want an upper bound, it would be enough to show that $\eb_i(S_{1\ldots j}) \geq \eb_i(S)$ for all $i, j \in [k]$.
This is not true for a general $S$, but it is true for some particular choices.
Recall that the zonotope $Z$ is defined as the convex hull of $Z_0= \{0, u_1\} + \{0, u_2\} +  \cdots + \{0, u_k\}$ in $\mathbb R^d$. 
We would like to take $S$ to be any scaled copy of $Z$, intersected with $\mathbb Z^d$.
For technical reasons it will be convenient to restrict to integer scale factors.
Write $Z(t) = (t\cdot Z) \cap \mathbb Z^d$ for $t \in \mathbb N$.

\begin{proposition}
For each $t \in \mathbb N$,
\begin{itemize}
\item[\rm{(i)}]  $Z(t+1) = (Z \cap \mathbb Z^d) + tZ_0$.  In particular, $Z(t) + Z_0 = Z(t+1)$.
\item[\rm(ii)] if $Z(t) \subseteq S \subset \mathbb Z^d$ and $S$ is finite, then $\eb(S) \geq \eb(Z(t))$.
\end{itemize}
\end{proposition}

\begin{proof}
(i) Certainly $(Z \cap \mathbb Z^d) + t Z_0 \subseteq Z(t+1)$.
For the reverse inclusion, let $v \in Z(t+1)$.
Then there are coefficients $\alpha_i$, $0 \leq \alpha_i \leq t+1$ such that $v = \sum_{i=1}^k \alpha_i u_i$.
For each $i$, write $\alpha_i = a_i + \beta_i$ where $0 \leq a_i \leq t$ is an integer and $0 \leq \beta_i \leq 1$. Then
\[
v = \sum_{i=1}^k a_i u_i + \sum_{i=1}^k \beta_i u_i \in tZ_0 + Z.
\]
Since $v \in \mathbb Z^d$ and $tZ_0 \subseteq \mathbb Z^d$, we in fact have $v \in tZ_0 + (Z \cap \mathbb Z^d)$. 

(ii) 
For each $i \in [k]$ and $v \in \mathbb Z^d$, write $L_{v,i} = \{v + \lambda u_i \, : \, \lambda \in \mathbb Z\}$ and $\mathcal L_i = \{L_{v,i} \,:\, v \in \mathbb Z^d\}$. We think of $\mathcal L_i$ as the set of `lines in direction $u_i$' (but note that, if the coordinates of $u_i$ have greatest common factor $h$, then $h$ distinct elements of $\mathcal L_i$ are contained in the same line in $\mathbb R^d$).
For any finite $S \subset \mathbb Z^d$, write $\mathcal L_i(S) = \{ L \in \mathcal L_i \,:\, L \cap S \neq \emptyset\}$. Every element of $\mathcal L_i(S)$ contributes at least $1$ to $\eb_i(S)$, so $\eb_i(S) \geq |\mathcal L_i(S)|$.
Moreover, equality holds when $S = Z(t)$ as every $L \in \mathcal L_i(Z(t))$ meets $Z(t)$ in an interval (that is, a set of the form $\{v + \lambda u_i \,:\, \lambda \in \mathbb Z, a \leq \lambda \leq b \}$ for some $v \in \mathbb Z^d$ and $a, b \in \mathbb Z$).
Thus whenever $Z(t) \subseteq S$, we have for each $i \in [k]$ that
\[
\eb_i(S) \geq |\mathcal L_i(S)| \geq |\mathcal L_i(Z(t))| = \eb_i(Z(t)). \qedhere
\]
\end{proof}

So we can use \eqref{telescope} to relate the edge boundary of $Z(t)$ to $|Z(t) + Z_0| - |Z(t)| = |Z(t+1)|-|Z(t)|$. We can understand this quantity using the following classical result (see for example \cite{BR07}).

\begin{theorem}[Ehrhart polynomials]\label{t:Ehrhart}
Let $P$ be a polytope with vertices in $\mathbb Z^d$.
Then for $t \in \mathbb N$,
\[
|t \cdot P \cap \mathbb Z^d| = \sum_{i=0}^d a_it^i,
\]
for some coefficients $a_i$ with $a_d = \vol(P)$.
\end{theorem}

\begin{proof}[Proof of Theorem~\ref{t:main}, upper bound]
Let $|Z(t)| = \sum_{i=0}^d a_it^i$ be the Ehrhart polynomial of $Z$.
Since $|Z(t)| = (1+o(1))\vol(Z)t^d$, we have
\[
t = (1+o(1))\left( \frac {|Z(t)|} {\vol(Z)} \right)^{1/d}.
\]
Hence, 
\begin{align*}
\eb(Z(t)) & = \sum_{i=1}^k \eb_i(Z(t)) \leq \sum_{i=1}^k \eb_i(Z(t) + \{0, u_1\} + \cdots + \{0, u_{i-1}\})
\\           & = \sum_{i=1}^k |Z(t) + \{0, u_1\} + \cdots + \{0, u_{i}\}| - |Z(t) + \{0, u_1\} + \cdots + \{0, u_{i-1}\}|
\\           & = |Z(t) + Z_0| - |Z(t)| = |Z(t+1)| - |Z(t)|
\\           & = \sum_{i=0}^d a_i ((t+1)^i - t^i)
               = (1+o(1))\vol(Z) d t^{d-1}
\\           & = (1+o(1))\vol(Z) d \left( \frac {|Z(t)|} {\vol(Z)} \right)^{(d-1)/d}
\\           & = (1+o(1)) d \vol(Z)^{1/d} |Z(t)|^{(d-1)/d},
\end{align*}
as $t$ grows large.
Now for $n \in \mathbb N$ with $n > |Z(1)|$, let $t$ be least with $|Z(t)| \geq n$.
Then
\begin{align*}
|Z(t)| - n \leq |Z(t)| - |Z(t-1)| \leq (1+o(1)) d \vol(Z)^{1/d} |Z(t-1)|^{(d-1)/d} = o(n),
\end{align*}
by the argument above, and so
\begin{align*}
\ebopt(n) & \leq \ebopt(|Z(t)|) \leq \eb(Z(t)) = (1+o(1)) d \vol(Z)^{1/d} |Z(t)|^{(d-1)/d} 
\\ & \leq (1+\delta) d \vol(Z)^{1/d} n^{(d-1)/d},
\end{align*}
for any $\delta > 0$ and $n$ sufficiently large, depending on $\delta$ and $Z$.
\end{proof}

\section{Open Problems} \label{s:open-problems}
Theorem \ref{t:main} gives an approximate answer to the edge isoperimetric problem on $G_{\mathcal{U}}$, saying that no shape can do much better than scalings of the zonotope $Z$.
However, as mentioned in the introduction, for many specific $\mathcal{U}$ much more is known. 
In particular, in all examples where a previous edge or vertex isoperimetric result is known, there is a nested sequence of optimal sets. In these cases, the optimal sets interpolate between integer scalings of the polytopes suggested by Theorems \ref{t:main} and \ref{t:ruzsa} by sequentially `filling in faces' of the optimal set. It is not inconceivable that a similar ordering could exist in every case.

\begin{question}
Let $\mathcal U = \{u_1, \ldots, u_k\}$ be a finite set of non-zero vectors that generate $\mathbb Z^d$ as a group. 
\begin{itemize}
\item Is there always an ordering $v_1, v_2, \ldots$ of $\mathbb Z^d$ such that $\eb_{G_{\mathcal U}}(\{v_1, \ldots, v_n\}) = \ebopt_{G_{\mathcal U}}(n)$?
\item Is there always an ordering $w_1, w_2, \ldots$ of $\mathbb Z^d$ such that $\partial_{v,G_{\mathcal U}}(\{w_1, \ldots, w_n\}) = \partial^*_{v,G_{\mathcal U}}(n)$?
\end{itemize}
\end{question}

Even if this isn't true, we could ask instead whether the optimal edge boundary is in fact achieved by the shapes giving the `natural' upper bounds, scalings of the zonotopes $Z$ for the edge boundary and the convex hulls $U$ for the vertex boundary. For the continuous analogues in Section \ref{s:continuous} it is known that equality holds for these shapes (and only these shapes).

\begin{question}
Let $\mathcal U = \{u_1, \ldots, u_k\}$ be a finite set of non-zero vectors that generate $\mathbb Z^d$ as a group.
Let $Z$ be the convex hull of $\{0, u_1\} + \cdots + \{0, u_k\}$ in $\mathbb R^d$, and let $Z(t) = (t\cdot Z) \cap \mathbb Z^d$. 
Let $U$ be the convex hull of $\mathcal U \cup \{0\}$ in $\mathbb R^d$, and let $U(t) = (t\cdot U) \cap \mathbb Z^d$.
\begin{itemize}
\item Is there an infinite sequence of $t_i \in \mathbb N$ such that $\eb_{G_{\mathcal U}}(Z(t_i)) = \ebopt_{G_{\mathcal U}}(|Z(t_i)|)$?
\item Is there an infinite sequence of $t_j \in \mathbb N$ such that $\partial_{v,G_{\mathcal U}}(U(t_j)) = \partial^*_{v,G_{\mathcal U}}(|U(t_j)|)$? 
\end{itemize}
\end{question}

Finally, for extremal problems in combinatorics it is often interesting to ask if a \emph{stability} result holds: if the boundary of a set is close to optimal, must the set be structurally close to some member of an optimal family? For example, there has recently been much interest in stability results for the edge isoperimetric inequality in the hypercube (see \cite{E11, EKL16, KL17a, KL17b}).
\begin{question}
Let $\mathcal U = \{u_1, \ldots, u_k\}$ be a finite set of non-zero vectors that generate $\mathbb Z^d$ as a group.
\begin{itemize}
\item If $\eb(S) \approx \ebopt(|S|)$ must $S$ be close to some $Z(t)$?
\item  If $\vb(S) \approx \vbopt(|S|)$ must $S$ be close to some $U(t)$?
\end{itemize}
\end{question}
We note that some stability results are known for the geometric results of Section \ref{s:continuous} (see for example \cite{F15}).

\section*{Acknowledgments} 
We would like to thank the anonymous referee for carefully reading an earlier draft of this paper and making a number of helpful comments.

\bibliographystyle{amsplain}


\begin{dajauthors}
\begin{authorinfo}[benbarber]
  Ben Barber\\
  School of Mathematics, University of Bristol and\\
  Heilbronn Institute for Mathematical Research, Bristol\\
  b.a.barber\imageat{}bristol\imagedot{}ac\imagedot{}uk \\
  \url{http://babarber.uk}
\end{authorinfo}
\begin{authorinfo}[joshuaerde]
  Joshua Erde\\
  Fachbereich Mathematik, Universit\"{a}t Hamburg\\
  Hamburg, Germany\\
  joshua.erde\imageat{}uni-hamburg\imagedot{}de
\end{authorinfo}

\end{dajauthors}

\end{document}